\documentclass[12pt,reqno]{amsart}
\usepackage[centering, includeheadfoot, hmargin=1.0in, tmargin=1.0in, bmargin=1in, headheight=7pt]{geometry}
\usepackage[alphabetic]{amsrefs}
\usepackage{url}
\usepackage{hyperref}
\hypersetup{colorlinks=true, allcolors=.}
\usepackage[dvipsnames]{xcolor}
\usepackage{booktabs}
\usepackage{graphicx}
\usepackage{color}
\usepackage{amsmath}
\usepackage[noabbrev,capitalize]{cleveref}
\usepackage{mathrsfs}
\usepackage{stmaryrd}
\usepackage{mathtools}
\usepackage{amssymb}
\usepackage[pagewise]{lineno}
\usepackage{appendix}
\usepackage{tikz-cd}
\usepackage{bbm}
\usepackage{tabu,multirow}
\usepackage{blkarray}
\usepackage{multicol}
\usepackage{wrapfig}

\makeatletter
\def\l@subsection{\@tocline{2}{0pt}{2.5pc}{5pc}{}}
\makeatother

\newtheorem*{thm*}{Theorem}
\newtheorem{lem}{Lemma}[section]
\newtheorem{thm}[lem]{Theorem}
\newtheorem{cor}[lem]{Corollary}
\newtheorem{prop}[lem]{Proposition}
\theoremstyle{definition}
\newtheorem{defn}[lem]{Definition}
\newtheorem{notation}[lem]{Notation}
 \newtheorem{assumption}[lem]{Assumption}

\newtheorem{examplex}[lem]{Example}
\newenvironment{example}
  {\pushQED{\qed}\examplex}
  {\popQED\endexamplex}

\newtheorem{conjecture}[lem]{Conjecture}

\theoremstyle{remark}
\newtheorem{remark}[lem]{Remark}
\numberwithin{table}{section}

\newcommand\<{\langle}
\renewcommand\>{\rangle}

\newcommand{\PP}{\mathbb P}

\newcommand{\ZZ}{\mathbb Z}

\newcommand{\NN}{\mathbb N}

\newcommand{\ba}{\mathbf a}
\newcommand{\bb}{\mathbf b}

\newcommand{\bbd}{\mathbf d}

\newcommand{\bbx}{\mathbf x}
\newcommand{\bby}{\mathbf y}

\DeclareMathOperator{\depth}{depth}

\DeclareMathOperator{\pdim}{pdim}

\DeclareMathOperator{\reg}{reg}

\newcommand{\Pnm}{\PP^n\times\PP^m}

\def\mustata{{Musta\c{t}\v{a}}}

\newcommand*{\defeq}{\mathrel{\vcenter{\baselineskip0.5ex \lineskiplimit0pt
                     \hbox{\scriptsize.}\hbox{\scriptsize.}}}%
                     =}

 \numberwithin{equation}{section}

\definecolor{bettergreen}{RGB}{50, 170, 50}

\definecolor{betterpink}{RGB}{227, 134, 211}

\begin{document}
%%%%%%%%%%%%%%%%%%%%%%%%%%%%%%%%%%%%%%%%%%%%%%%%%%%
\vspace*{-1cm}

\title{On virtual resolutions of points in a product of projective spaces}

\author[Bailly-Hall]{Isidora Bailly-Hall}
\address{Grinnell College}
\email{baillyha@grinnell.edu}

\author[Berkesch]{Christine Berkesch}
\address{University of Minnesota}
\email{cberkesc@umn.edu}

\author[Dovgodko]{Karina Dovgodko}
\address{Columbia University}
\email{kmd2235@columbia.edu}

\author[Guan]{Sean Guan}
\address{University of California, Berkeley}
\email{seanguan@berkeley.edu}

\author[Sivakumar]{Saisudharshan Sivakumar}
\address{University of Florida}
\email{sivakumars@ufl.edu}

\author[Sun]{Jishi Sun}
\address{University of Michigan}
\email{joshsun@umich.edu}

%\date{}

\begin{abstract}
%For JMM: 
%Free resolutions, or syzygies, with a graded structure are algebraic objects that encode many geometric properties. This correspondence lies at the heart of classical projective algebraic geometry. Virtual resolutions were recently introduced by Berkesch, Erman, and Smith to produce a similar correspondence for smooth toric varieties. We will describe two methods for producing nice virtual resolutions for a  finite sets of points in $\mathbb{P}^n \times \mathbb{P}^m$. We first extend to $\mathbb{P}^n \times \mathbb{P}^m$ a result of Harada, Nowroozi, and Van Tuyl for $\mathbb{P}^1 \times \mathbb{P}^1$ by intersecting with a sufficiently high power of one set of variables. Then, we describe an explicit virtual resolution for a set $X$ of sufficiently general position in $\mathbb{P}^1 \times \mathbb{P}^2$; this is a subcomplex of the free resolution for $X$. Along the way, we provide an explicit relationship between Betti numbers and higher difference matrices of Hilbert functions. 
%For report:
For finite sets of points in $\Pnm$, we produce short virtual resolutions, as introduced by Berkesch, Erman, and Smith~\cite{virtual-original}. 
We first intersect with a sufficiently high power of one set of variables for points in $\Pnm$ to produce a virtual resolution of length $n+m$. 
Then, we describe an explicit virtual resolution of length $3$ for a set of points in sufficiently general position in $\PP^1 \times \PP^2$, via a subcomplex of a free resolution. 
This first result generalizes to $\Pnm$ work of Harada, Nowroozi, and Van Tuyl~\cite{HNV-points} and  the second partially generalizes work of \cite{HNV-points} and Booms-Peot~\cite{booms-peot}, which were both for $\PP^1\times\PP^1$. 
Along the way, we also note an explicit relationship between Betti numbers and higher difference matrices of bigraded Hilbert functions for $\Pnm$. 
\end{abstract}
\maketitle
%\tableofcontents

\vspace*{-2mm}
%%%%%%%%%%%%%%%%%%%%%%%%%%%%%%%%
\section{Introduction}
%%%%%%%%%%%%%%%%%%%%%%%%%%%%%%%%
\label{sec:introduction}

Minimal free resolutions of the vanishing ideal of an embedded projective variety $Y$ contain important geometric information; for example, the length of such a complex is bounded below by the codimension of $Y$ and above by the dimension of the ambient projective space. 
When working over the Cox ring of a smooth projective toric variety, the virtual resolutions introduced by Berkesch, Erman, and Smith in \cite{virtual-original}, can be shorter than minimal free resolutions and encode more geometry; for one, it is possible to find such complexes with length at most  the dimension of the ambient space, see~\cites{virtual-original,FH,HHL,short-HST,reggie-thesis}. 
There is growing interest in producing families of short virtual resolutions, see~\cites{virtualJSAG,duarte-seceleanu,virtualCM,reu2018,loper,yang-monomial,booms-cobb,HNV-points,reu2019,mahrud-resDiag,booms-peot,virtual-cellular}. 
Most relevant to this article are \cites{HNV-points,booms-peot}, which study virtual resolutions of points in $\PP^1 \times \PP^1$, as we will consider finite sets of points $X \subseteq \Pnm$. 

Let $S=\Bbbk[x_0, x_1,\ldots, x_n,y_0, y_1,\ldots, y_m]$ denote the \emph{Cox ring} of $\Pnm$, where $\Bbbk$ is an algebraically closed field, and $S$ carries the multigrading 
$\deg(x_i) = (1,0)\in\ZZ^2$ and $\deg(y_j) = (0,1)\in\ZZ^2$ for all $i,j$. 
Let $\bbx = x_0, x_1, \ldots, x_n$ and $\bby = y_0, y_1, \ldots, y_m$. 
The \emph{irrelevant ideal} of $\Pnm$ in $S$ is 
\[
B:= \< x_0, x_1,\ldots, x_n \> \cap \< y_0,y_1, \ldots, y_m \> = \<\bbx\>\cap\<\bby\>.
\]
Let $X$ denote a finite set of points in $\Pnm$, and let $I_X$ denote the homogeneous vanishing ideal of $X$ in $S$. 
A graded free complex of $S$-modules 
    \[
    F_\bullet 
    \ \coloneqq \ 
    F_0\gets F_1\gets F_2\gets \cdots
    \] 
is a \emph{virtual resolution} of a finitely generated $\ZZ^2$-graded $S$-module $M$ if it satisfies: 
\begin{enumerate}
    \item For each $i>0$, there is some $t$ such that $B^tH_i(F_\bullet) = 0$, and
    \item $H_0(F_\bullet)/\Gamma_B(H_0(F_\bullet))\cong M/\Gamma_B(M)$, 
    \end{enumerate}
where $
    \Gamma_B(M)\coloneqq \{m\in M\mid B^tm = 0\text{ for some } t\in\NN\}$. 
Equivalently, $F_\bullet$ is a virtual resolution of $M$ if $\widetilde{F_\bullet}$ is a locally free resolution of $\widetilde{M}$ over $\Pnm$. 

For a finite set $X$ of points in $\PP^1\times\PP^1$, \cite{HNV-points}*{Theorem 4.2} made explicit the construction of a virtual resolution for $S/I_X$ in \cite{virtual-original}*{Theorem 4.1}. 
Our first main result generalizes this to $\Pnm$. 

\begin{thm}
\label{thm:intersectionPPnxPPm}
Let $X$ be a finite set of points in $\Pnm$ with natural first projection $\pi_1\colon \Pnm\to\PP^n$. 
Let $\ell = |\pi_1(X)|$ denote the number of unique first coordinates among the points in $X$. 
For all $t\geq \ell-1$, the minimal free resolution of $S/(I_X\cap\<\bbx\>^t)$ is a virtual resolution of $S/I_X$ of length $n+m$.
\end{thm}

A second method for finding virtual resolutions, called \emph{virtual of a pair} in \cite{virtual-original} obtains virtual resolutions by taking certain subcomplexes of the minimal free resolution, see \cref{thm:trimming}. 
\cite{HNV-points}*{Theorem 3.1} used the virtual of a pair construction to give an explicit description of certain virtual resolutions for finite sets of points $X \subseteq \PP^1 \times \PP^1$. 
In the same setting \cite{booms-peot}*{Theorems 3.3, 1.4} built on this result and gave a sufficient condition for when virtual resolutions of a pair have length $2$.

Our second main result generalizes part of this work to $\PP^1\times\PP^2$, while relying on a weakened form of the Minimal Resolution Conjecture (see \cref{conj:mrc}), which is open for multiprojective spaces. 

\begin{defn}
\label{def:suffGeneral}
For a finite set of points $X\subseteq \Pnm$, 
the \emph{Hilbert matrix} $H_X$ has entries 
\[
(H_X)_{i,j} \defeq H_{S/I_X}(i,j) 
\quad\text{for all }
(i,j)\in\ZZ^2
\]
tabulating the bigraded Hilbert function of $S/I_X$.
If 
\[
H_X(i,j)=\min\{|X|,T_{i,n}T_{j,m}\} 
\quad\text{for all}\quad 
i,j\geq 0,
\]
with
\[
T_{a,b} := \binom{a+b}{b}, 
\]
then we say that $X$ has a \emph{generic Hilbert matrix}. 
\cref{prop:generic} shows that this is an open condition on the Hilbert scheme of points in $\Pnm$. 
We will further say that $X$ is a set of points in \emph{sufficiently general position} if it has a generic Hilbert function and satisfies an additional condition on its first Betti numbers, as given in \cref{conj:mrc}.   
\end{defn} 

\begin{thm}[See \cref{thm:virtual-of-pair-stronger} and \cref{sec:cases-2-to-11}]
\label{thm:virtual-of-a-pair}
Let $X\subseteq \PP^1\times\PP^2$ be a finite set of points in sufficiently general position. 
If a weakened form of the Minimal Resolution Conjecture holds (see \cref{conj:mrc}), 
then $S/I_X$ has a virtual resolution of length $3$ obtained from the virtual of a pair construction at $\bbd=(|X|-1,0)$.
\end{thm}

It is well-known that Hilbert functions can be computed as alternating sums of graded Betti numbers. 
For $\PP^1\times\PP^1$, \cite{giuffrida-92}*{Proposition 3.3} show that Betti numbers can be expressed via a certain ``second difference matrix" of $H_X$, and 
the authors of \cite{HNV-points} and \cite{booms-peot} exploited this in their study of virtual resolutions for sets of points in $\PP^1\times\PP^1$. 
In our proof of \cref{thm:virtual-of-a-pair}, we similarly use an analogous observation that for $\Pnm$. 

\begin{defn}
\label{def:column-row-difference-operators}
Define a partial order on $\ZZ^2$ by letting $(i,j)\geq (i',j')$ if $i\geq i'$ and $j\geq j'$. For any infinite matrix $H = (H_{i,j})$ indexed by $(i,j)\in\ZZ^2$, 
the \emph{column difference operator} $\Delta^C$ and \emph{row difference operator} $\Delta^R$ are
\begin{align*}
\Delta^C(H)_{i,j} := H_{i,j} - H_{i-1,j}
\qquad\text{and}\qquad 
\Delta^R(H)_{i,j} := H_{i,j} - H_{i,j-1}. 
\end{align*}
Note that the operators $\Delta^C$ and $\Delta^R$ commute. 
\end{defn}

\begin{prop}
\label{prop:CRgeneral}
If $M$ is a finitely generated bigraded $S$-module with minimal free resolution 
\[
F_\bullet:\quad \bigoplus_{i,j\geq 0}S(-i,-j)^{\beta_{0,(i,j)}}\gets \bigoplus_{i,j\geq 0}S(-i,-j)^{\beta_{1,(i,j)}} \gets \cdots\gets \bigoplus_{i,j\geq 0}S(-i,-j)^{\beta_{k,(i,j)}} \gets \cdots, 
\]
then the Hilbert matrix $H_M$ satisfies
\[
\big((\Delta^C)^{n+1}(\Delta^R)^{m+1}H_M\big)_{i,j}=B_{i,j}, 
\]
where 
    $B_{p,q} \defeq \sum_{k=0}^{\infty} (-1)^k\beta_{k,(p,q)}$ 
for all $(p,q)\in\ZZ^2$. 
\end{prop}

%%%%%%%%%%%%%%%%%%%%%%%
\subsection*{Outline}
%%%%%%%%%%%%%%%%%%%%%%%
We examine syzygies of points in $\Pnm$ in  \cref{sec:intersectingPPnxPPm} and prove \cref{thm:intersectionPPnxPPm}. 
\cref{prop:CRgeneral} is proven in  \cref{sec:MatrixandBetti}. 
We then focus our attention on $\PP^1\times\PP^2$, discussing our notion of points in sufficiently general position in \cref{sec:betti-points-mrc}, including a weakened statement of the Minimal Resolution Conjecture. 
Finally, we use this to prove \cref{thm:virtual-of-a-pair} in \cref{sec:VResPair}.

%%%%%%%%%%%%%%%%%%%%%%%%%%%%%%%%%%%%
\subsection*{Acknowledgements}
%%%%%%%%%%%%%%%%%%%%%%%%%%%%%%%%%%%%
This work is a result of the 
University of Minnesota School of Mathematics REU in Algebra and Combinatorics during Summer 2023, funded by NSF RTG Grant DMS-1745638. 
The group is grateful for TA Sasha Pevzner's guidance during the REU. 
We also thank Daniel Erman for helpful conversations related to this work and Gregory G. Smith for supplying a sketch of \Cref{prop:generic}. 
CB was partially supported by NSF Grant DMS 2001101.

%%%%%%%%%%%%%%%%%%%%%%%%%%%%%%%%%%%%%%%%%%%
\section{A short virtual resolution via intersection in \texorpdfstring{$\Pnm$}{PPn x PPm}}
%%%%%%%%%%%%%%%%%%%%%%%%%%%%%%%%%%%%%%%%%%%
\label{sec:intersectingPPnxPPm}

By the end of this section, we prove \cref{thm:intersectionPPnxPPm}.
We begin with several preliminary results. First, we note that after a change of coordinates, we may assume that $y_0$ is a nonzero divisor on $S/I_X$; see \cref{dodging}. 

For a point 
\[
\ba \times \bb 
= [1: a_1: a_2:\cdots:a_n] 
\times [1: b_1:b_2:\cdots: b_m]
\in \Pnm,
\]
and for $i,j>0$, let $L_{a_i} \defeq a_ix_0-x_i$ and $L_{b_j} \defeq b_jy_0-y_j$. Let 
\[
I_{\ba}=\< L_{a_1}, L_{a_2},\ldots, L_{a_n}\> \text{ and } 
I_{\bb}=\< L_{b_1}, L_{b_2},\ldots, L_{b_m}\>,
\]
so that $I_{\ba}+I_{\bb}$ is the bihomogeneous ideal for the point $\ba\times\bb$.

\begin{assumption}
\label{dodging}
Throughout this paper, we will consider a finite set of points $X \subseteq \Pnm$. 
After a linear change of coordinates, we may and will assume that every point in $X$ is of the form $[1: a_1: \cdots:a_n] \times [1:b_1: \cdots: b_m]$. 
\end{assumption}

\begin{notation}
\label{not:pointsNotation}
Let $\pi_1(X) = \{P_1, P_2, \ldots, P_{\ell}\}$ denote the set of distinct first coordinates in $X$, so $\ell = |\pi_1(X)|$.
Let $X_k \defeq \pi_1^{-1}(P_k) \cap X$ denote the pre-image in $X$ of $P_k$, and write 
\[
X_k = \{P_k \times Q_{k,1}, P_k \times Q_{k,2}, \ldots, P_k\times Q_{k,r_k} \} 
\subseteq X,
\]
where $r_k = |\pi_1^{-1}(P_k)|$. 
Thus, for each $k$ with $1 \le k \le \ell$, 
\begin{align}
\label{eq:IXk}
I_{X_k}=I_{P_k}+\bigcap_{j=1}^{r_k} I_{Q_{k,j}}.
\end{align}
Note that the ideal $I_X$ is necessarily $B$-saturated. 
For each $j\in\{1,2,\dots,\ell\}$, 
let 
\[
J_{j} \defeq \bigcap_{k=1}^j I_{X_k}.
\]
Note that $J_{j} = I_{Y_j}$, where $Y_j = \bigsqcup_{k=1}^{j}X_k$ for $1 \leq j \leq \ell$.
\end{notation}

\begin{lem}
\label{lem:contains-x-ideal}
Let $X$ be a set of at least two points in $\Pnm$ with $\ell=|\pi_1(X)|$ and $J_{\ell-1}$ as in \cref{not:pointsNotation}. 
If $t \geq \ell -1$, then 
$\< \bbx  \>^t \subseteq I_{X_{\ell}} + J_{\ell-1}$.
\end{lem}
\begin{proof}
We prove this statement by induction on $\ell = |\pi_1(X)|$. 
The base case $\ell=1$ is trivial. 
If $\ell = 2$, then 
$I_{X_1} = \< L_{a_1}, \ldots, L_{a_n}, G_1, G_2,\ldots, G_k \>$, 
where $L_{a_i}$ is a linear form in $x_0$ and $x_i$, and each $G_i$ is a polynomial in the $\bby$-variables; 
the ideal $I_{X_2} = \< L_{a_1}', \ldots, L_{a_n}', G_1', G_2', \ldots, G_{k'}' \>$ is of the same form. 
Then together, 
\[
I_{X_1} + I_{X_2} = \< x_0,x_1,\dots,x_n, G_1, G_2, \ldots, G_k, G_1', G_2', \ldots, G_{k'}' \>,
\]
which contains $\<\bbx\>^t$ for all $t\geq 1$, as desired. 

Now for any $\ell > 2$, write $X = Y_1 \sqcup Y_2 \sqcup X_{\ell}$,  
where 
$Y_1 = \bigsqcup_{k=1}^{\ell-2}{X_k}$ 
and $Y_2 = {X_{\ell-1}}$, 
so $|\pi_1(Y_1)| = \ell-2$ and $|\pi_1(Y_2)| = 1$.
Then by the inductive hypothesis, 
\[
\<\bbx \>^{\ell-2} 
 \subseteq I_{X_{\ell}} + I_{Y_1}
\qquad \text{and}\qquad 
\< \bbx \> 
 \subseteq I_{X_{\ell}} + I_{Y_2}.
\]
Putting these together yields the desired result: 
\begin{align*}
\<\bbx\>^{\ell-2} \cdot \<\bbx\> 
 \subseteq
(I_{X_{\ell}} + I_{Y_1}) \cdot (I_{X_{\ell}} + I_{Y_2}) 
 \subseteq I_{X_{\ell}} + \left(I_{Y_1} \cap I_{Y_2}\right)
 = I_{X_{\ell}} + \bigcap_{k=1}^{\ell-1}I_{X_k}
= I_{X_{\ell}} + J_{\ell-1}.
\qquad\quad\,\qedhere
\end{align*}
\end{proof}

We now begin working towards \cref{prop:primaryDecompPPnxPPm}, which provides a primary decomposition of $\left\<I_X\cap\<\bbx\>^t,y_0\right\>$, the final ingredient needed to prove \cref{thm:intersectionPPnxPPm}. 

\begin{lem}
\label{lem:NZD-SEQ}
If $I$ is an ideal in $S$ such that $S/I$ has a nonzero divisor $L$ of degree $(0,1)$, then for all $(i,j)\in\ZZ^2$, 
\[
H_{S/(I + \< L \>)}(i,j) 
= H_{S/I}(i,j) - H_{S/I}(i,j-1).
\]
\end{lem}
\begin{proof}
There is a graded short exact sequence 
\[
0\to \frac{S}{I : \< L \>} (0, -1)
\xrightarrow{\ \ \cdot \overline{L}\ \ } 
\frac{S}{I} \longrightarrow \frac{S}{I+\<L\>} \to 0,  
\]
and, since $L$ is a nonzero divisor on $S/I$,  
$I : \< L \> = I$, so the desired equality of Hilbert functions follows. 
\end{proof}

\begin{lem}
\label{lem:VSeq}
Let $X$ be a finite set of points in $\Pnm$ as in \cref{dodging}. If $i \ge |\pi_1(X)|-1$, then there is an equality of bigraded pieces for all $j\in\ZZ$: 
\begin{equation}
\label{eq:VS}
\left[
    \bigcap_{P_k\in \pi_1(X)} \< I_{X_k}, y_0 \>
\right]_{(i,j)}
= \left[ \< I_X , y_0 \> \right]_{(i,j)}.
\end{equation}
\end{lem}
\begin{proof}
Because the vector space on the right-hand side of \eqref{eq:VS} is contained in the left-hand side, it suffices to show that these vector spaces have the same dimension, which we use induction on $\ell\defeq |\pi_1(X)|$. 
The base case $\ell=1$ is a tautology for all $i$. 

By way of induction, suppose that $\ell > 1$ and \eqref{eq:VS} holds for all sets of points with strictly fewer distinct first coordinates than $X$. 
To examine the right-hand side of \eqref{eq:VS}, note that \cref{dodging} and \cref{lem:NZD-SEQ} together imply that 
\begin{equation}
\label{eq:hilb-y0}
H_{S/\< I_X, y_0\>}(i,j)=H_{S/I_X}(i,j)-H_{S/I_X}(i,j-1).
\end{equation}
Since $I_X = \bigcap_{k=1}^{\ell} I_{X_k}$, there is a graded short exact sequence
\[
0 \to S/I_X \to S/J_{\ell-1}
\oplus S/I_{X_{\ell}}
\to S/(J_{\ell-1}+I_{X_{\ell}}) \to 0, 
\]
where $J_{\ell-1}$ is as in \cref{not:pointsNotation}. 
By \cref{lem:contains-x-ideal}, the ideal $J_{\ell -1 } + \< I_{X_{\ell}} \>$ contains $\<\bbx \>^{\ell - 1}$. 
Hence   
$H_{S/(J_{\ell-1}+I_{X_{\ell}})}(i,j)=0$ for all $i\geq \ell-1$, 
so by the additivity of Hilbert functions on graded short exact sequences, for all $j$, 
\begin{align}
\label{eq:hilb-ijell}
H_{S/I_X}(i,j) 
&= H_{S/J_{\ell-1}}(i,j) + H_{S/I_{X_{\ell}}}(i,j) 
\\
\text{and}\quad 
H_{S/I_X}(i,j-1)
&= H_{S/J_{\ell-1}}(i,j-1) + H_{S/I_{X_{\ell}}}(i,j-1).
\label{eq:hilb-ij-1ell}
\end{align}
Combining \eqref{eq:hilb-y0}, \eqref{eq:hilb-ijell}, and \eqref{eq:hilb-ij-1ell} yields the following expression for the Hilbert function of the right-hand side of \eqref{eq:VS}:  
\begin{multline}
\label{eq:RHS}
H_{S/\< I_X, y_0\>}(i,j)=
H_{S/J_{\ell-1}}(i,j)+H_{S/I_{X_{\ell}}}(i,j)
-H_{S/J_{\ell-1}}(i,j-1)-H_{S/I_{X_{\ell}}}(i,j-1).
\end{multline}

Now to examine the left-hand side of \eqref{eq:VS}, 
it follows from the inductive hypothesis that 
\[
\left[\bigcap_{k=1}^{\ell-1} \< I_{X_k}
,  y_0\>\right]_{(i,j)}
= \left[\bigcap_{k=1}^{\ell-1} I_{X_k}
+ \< y_0\>\right]_{(i,j)}.
\]
Since there is a graded short exact sequence 
\[
0\to 
\frac{S}{\bigcap_{k=1}^{\ell} \< I_{X_k}, y_0 \>} 
\to 
\frac{S}{ J_{\ell-1}+\< y_0 \>} 
\oplus 
\frac{S}{I_{X_{\ell}}+ \< y_0\>}
\to
\frac{S}{J_{\ell-1}+ I_{X_{\ell}} + \< y_0\>} 
\to 0, 
\]
the following equality of Hilbert functions holds for all $j$: 
\begin{multline*}
H_{S/\bigcap_{k=1}^{\ell}\< I_{X_k}, y_0 \>}(i,j) 
= 
H_{S/(J_{\ell-1}+\< y_0 \>)}(i,j) 
+ H_{S/(I_{X_{\ell}}+\< y_0 \>)}(i,j) 
- H_{S/(J_{\ell-1}+I_{X_{\ell}}+\< y_0 \>)}(i,j). 
\end{multline*}
Next, by \cref{lem:contains-x-ideal}, since $i \geq |\pi_1(X)| - 1$,  there is a containment 
$\< \bbx \>^i \subseteq J_{\ell-1} + I_{X_{\ell}}$, so 
\[
S_{i,j} = 
\left[ 
    \bigcap_{k=1}^{\ell-1}  I_{X_k} + I_{X_{\ell}} 
\right]_{(i,j)}
\qquad \text{and} \qquad 
H_{S/(J_{\ell-1}+I_{X_{\ell}}+\< y_0 \>)}(i,j) 
= 0.
\]
Since $y_0$ is a nonzero divisor on $S/J_{\ell-1}
$ and $S/I_{X_\ell}$, by \cref{lem:NZD-SEQ}, for all $j$, 
\begin{align*}
H_{S/(J_{\ell-1}+\< y_0 \>)}(i,j)
&= H_{S/J_{\ell-1}}(i,j) - H_{S/J_{\ell-1}}(i,j-1)
\\
\text{and}\qquad 
H_{S/(I_{X_{\ell}}+\< y_0 \>)}(i,j) 
&= H_{S/I_{X_{\ell}}}(i,j) - H_{S/I_{X_{\ell}}}(i,j-1).
\end{align*}
Comparing with \eqref{eq:RHS}, it follows that for $i\geq \ell-1$, 
\[
H_{S/\< I_X, y_0\>}(i,j) = 
H_{S/\bigcap_{k=1}^{\ell}\< I_{X_k}, y_0 \>}(i,j), 
\]
so \eqref{eq:VS} is an equality, as desired.
\end{proof}

When a set of points $X$ has a generic Hilbert function, equality can be obtained sooner. 

\begin{cor}
\label{cor:VSgp}
If $X$ is a finite set of points with generic Hilbert function in $\Pnm$ as in \cref{dodging} and $d = \min\{r \mid T_{r,n} \ge |X|\}$, then \eqref{eq:VS} holds for all $i\geq d$ and $j\in\ZZ$. 
\end{cor}
\begin{proof}
    First consider the case that $j=0$. Here,  
    \[
    \left[
    \bigcap_{P_k\in \pi_1(X)} \< I_{X_k}, y_0 \>
    \right]_{(i,0)} 
    = [I_X]_{(i,0)}
    \qquad\text{and}\qquad
    \left[ \< I_X , y_0 \> \right]_{(i,0)}=[I_X]_{(i,0)}, 
    \]
    so \eqref{eq:VS} holds. 
    When $j \ge 1$, just as in  \cref{lem:VSeq}, because the vector space on the right-hand side (RHS) of \eqref{eq:VS} is contained in the left-hand side (LHS), it suffices to show that for all $i \ge d$, the dimension of the RHS is is greater than or equal to the dimension of the LHS. This is equivalent to showing that
    \[
    \dim_{\Bbbk}\left[
    \bigcap_{P_k\in \pi_1(X)} \< I_{X_k}, y_0 \>
    \right]_{(i,j)} 
    \le \dim_{\Bbbk}\left[ \< I_X , y_0 \> \right]_{(i,j)}.
    \]
    In turn, this is equivalent to showing
    $H_{RHS}(i,j) \le H_{LHS}(i,j)$. 
    By \cref{dodging} and \cref{lem:NZD-SEQ}, it follows that 
    \[H_{RHS}(i,j)=H_{S/\<I_X,y_0\>}(i,j)=H_{S/I_X}(i,j)-H_{S/I_X}(i,j-1).\]
    Since the points are in sufficiently general position and $i \ge d$, 
    $H_{S/I_X}(i,j)=H_{S/I_X}(i,j-1)=|X|$, so $H_{RHS}(i,j)=0$.
\end{proof}

\begin{prop}
\label{prop:primaryDecompPPnxPPm}
Let $X$ be a finite set of points in $\Pnm$ as in \cref{dodging} and \cref{not:pointsNotation}. 
Then for any integer $t\geq |\pi_1(X)|-1$, there is a primary decomposition of the form 
\begin{equation*}
\< I_X \cap \< \bbx \>^t, y_0 \> 
= \bigcap_{P_k\in \pi_1(X)} \< I_{X_k}, y_0 \> 
  \cap 
  \< \< \bbx \>^t, y_0 \>.
\end{equation*}
\end{prop}

\begin{remark}
Note that when $m=1$, then the statement of \cref{prop:primaryDecompPPnxPPm} can be made more explicit, since for each point $\ba\times\bb \in X$, $I_{\ba\times\bb}=\< L_{a_{1}}, \ldots L_{a_n}, L_{b_1}\>$. 
Then by \eqref{eq:IXk}, for any $P_k\in \pi_1(X)$, 
\[
I_{X_k} = \<L_{P_{k,1}}, L_{P_{k,2}},\ldots, L_{P_{k,n}}, L_{Q_{k,1}}  L_{Q_{k,2}} \cdots L_{Q_{k,r_k}}\>.
\]
Since each $L_{Q_{k,j}}$ is a linear form in $y_0$ or $y_1$, their product will be a polynomial of degree $r_k$ in $y_0$ and $y_1$. By \cref{dodging}, each $Q_{k,j}$ is of the form $[b_{j,0}:b_{j,1}]$, with $b_{j,0}=1$. Therefore, the product of the $L_{Q_{k,j}}$'s will have a term $y_1^{r_k}$. Thus
\[
\<I_{X_k},y_0\> = \<L_{P_{k,1}}, L_{P_{k,2}}, \ldots, L_{P_{k,n}}, y_0, y_1^{r_k}\>.
\]
\end{remark}

\begin{proof}[Proof of \cref{prop:primaryDecompPPnxPPm}]
For any $S$-ideals $I_1,I_2,I_3$, $(I_1\cap I_2) + I_3 \subseteq (I_1+I_3)\cap (I_2+I_3)$; in particular, for all $t\geq 0$, 
\begin{equation}
\label{eq:decomp-new-1}
\< I_X \cap \< \bbx \>^t, y_0 \> 
\subseteq 
\bigcap_{P_k\in \pi_1(X)} \< I_{X_k}, y_0 \> 
    \cap \left\< \< \bbx \>^t, y_0 \right\>. 
\end{equation}
We will now show that when $t\geq |\pi_1(X)|-1$, the left-hand side of \eqref{eq:decomp-new-1} contains the right-hand side, and it suffices to do so on each bigraded piece, i.e., show that for all $(i,j)\in \ZZ^2$ and $t\geq |\pi_1(X)|-1$, 
\begin{equation}
\label{eq:decomp-new-2}
\left[
    \bigcap_{P_k\in \pi_1(X)} \< I_{X_k}, y_0 \>
    \cap 
    \left\< \< \bbx \>^t, y_0 \right\> 
\right]_{(i,j)}
\subseteq
\left[
    \left\< I_X \cap \< \bbx \>^t, y_0 \right\>
\right]_{(i,j)}.
\end{equation}
We consider two cases, based on a comparison between $i$ and $t$. 
First, $0\leq i<t$ 
and suppose that $f$ is in the left-hand side of \eqref{eq:decomp-new-2}; in particular, 
$f\in \left[\< \<\bbx\>^t, y_0 \>\right]_{(i,j)}$. 
Since $i<t$ by assumption, $f$ is of low degree, so it must be that $f\in\<y_0\>$. 
This clearly also places $f$ in the right-hand side of \eqref{eq:decomp-new-2}, so the equality holds in this case.

Second, if $i\geq t$, then $[\<\bbx\>^t]_{i,j} = S_{(i,j)}$, so 
\[
\left[
    \left\< \<\bbx\>^t, y_0 \right\> 
\right]_{(i,j)} 
= S_{(i,j)}
\qquad 
\text{and}
\qquad 
\left[
    \< I_X \cap \< \bbx \>^t, y_0 \>
\right]_{(i,j)} 
= \left[ \< I_X,y_0 \> \right]_{(i,j)}.
\]
Thus it suffices to show that for $i\geq t$, 
\begin{equation*}
\left[\bigcap_{P_k\in \pi_1(X)} \< I_{X_k}, y_0 \>\right]_{(i,j)}
\subseteq
\left[\< I_X , y_0 \>\right]_{(i,j)}, 
\end{equation*}
which was shown in \cref{lem:VSeq}, finishing this case. 
Having established both cases, we have now shown that \eqref{eq:decomp-new-1} is an equality for $t\geq|\pi_1(x)|-1$. 

It now remains to show that the ideals on the right-hand side of \eqref{eq:decomp-new-1} are primary. 
By \cite{monomial-algebras-book}*{Proposition 6.1.7}, $\left\< \<\bbx\>^t, y_0 \right\>$ is primary. 
To see that $\< I_{X_k},y_0 \>$ is primary for each $1\leq k\leq |\pi_1(X)|$, recall from \cref{not:pointsNotation} that for each $k$, 
$X_k = \{P_k \times Q_{k,1}, \ldots, P_k \times Q_{k,r_k}\}$ 
and 
\begin{align*}
    \< I_{X_k}, y_0\> 
    \ = \ I_{P_k} + \bigcap_{j=1}^{r_k} I_{Q_{k,j}} + \< y_0 \>
    \ = \ \left\< L_{P_{k,1}}, L_{P_{k,2}}, \ldots, L_{P_{k,n}}, G_1, G_2,\ldots, G_s, y_0\right\>, 
\end{align*}
where $L_{P_{k,i}} = P_{k,i}x_0-x_i$ 
(with $P_{k,i}$ allowed to be $0$) 
are the linear forms generating each $I_{P_k}$ 
and $G_1,G_2,\dots,G_s$ are forms in the $\bby$ such that
$\< G_1,G_2,\dots,G_s \> 
= \bigcap_{j=1}^{r_k} I_{Q_{k,j}}$. 
Putting this together, 
\begin{multline*}
S/\< I_{X_k}, y_0\> 
= 
S/\left\< L_{P_{k,1}}, L_{P_{k,2}}, \ldots, L_{P_{k,n}}, G_1,G_2,\dots,G_s, y_0\right\> 
\\
\cong \Bbbk[x_0,y_0, y_1, \ldots, y_m]/\left\< G_1,G_2,\dots,G_s, y_0\right\>,
\end{multline*}
so it suffices to show that the ideal 
$J \defeq \< G_1,G_2,\dots,G_s, y_0\>$ is primary in $\Bbbk[x_0,y_0,\dots,y_m]$. 
To this end, let $J'$ denote the ideal generated by the same elements, but viewed as an ideal of $\Bbbk[y_0,\dots,y_m]$. 
Since $J = J'\Bbbk[x_0,y_0,\dots,y_m]$, it suffices to show that $J'$ is primary in $\Bbbk[y_0,\dots,y_m]$. 
Now by \cref{dodging}, 
$Q_{k,j}=[1:Q_{k,{j_1}}: \ldots:Q_{k,{j_m}}]$, 
so  
\[
I_{Q_{k,j}}=\< Q_{k,{j_1}}y_0-y_1, \ldots, Q_{k,{j_m}}y_0-y_m\>.
\]
Combining this with 
\[
\left\< \prod_{j=1}^{r_k} I_{Q_{k,j}}, y_0 \right\>\subseteq \left\< \bigcap_{j=1}^{r_k}I_{Q_{k,j}}, y_0 \right\> = J', 
\]
it follows that 
$\< y_0, \ldots, y_m\>^{r_k} 
\subseteq J' 
\subseteq \< y_0, \ldots, y_m\>$, 
so the radical of $J'$ is $\<\bby\>$, which is maximal in $\Bbbk[y_0, \ldots, y_m]$. 
Thus $J'$ is primary in $\Bbbk[y_0, \ldots, y_m]$, as desired.
\end{proof}

We are now prepared to prove the main result of this section.  

\begin{proof}[Proof of \cref{thm:intersectionPPnxPPm}]
Since the ideal $I_X$ is $B$-saturated, $I_X\cap \< \bbx\>^t : B^\infty = I_X$. 
Thus a minimal free resolution of $S/(I_X\cap \< \bbx\>^t)$ is indeed a virtual resolution of $S/I_X$.

Next, note that by \cite{virtual-original}*{Proposition~2.5.a}, the minimal free resolution of $S/(I_X\cap \<\bbx\>^t)$ has length at least $n+m$, the codimension of $I_X$. 
It is thus left to show that the minimal free resolution of $S/(I_X\cap \<\bbx\>^t)$ is of length at most (and thus equal to) $n+m$. 
To begin, by the Auslander--Buchsbaum formula, 
\[
\pdim (S/(I_X\cap \< \bbx\>^t) 
= 
n+m+2 - \depth(S/(I_X\cap \< \bbx\>^t),
\] 
thus it suffices to show that the depth of $S/(I_X\cap \< \bbx\>^t)$ is at least $2$. 
To do this, we produce a regular sequence on $S/(I_X\cap \<\bbx\>^t)$ of length two. 

Without loss of generality by \cref{dodging}, $y_0$ is a nonzero divisor on $S/I_X$. 
Further, we claim that $y_0$ is a nonzero divisor on $S/(I_X\cap \< \bbx\>^t)$. 
Indeed, if $f\in S$ with $y_0f\in I_X\cap\< \bbx\>^t$, 
then $y_0f\in I_X$ and $y_0f\in \< \bbx\>^t$. 
This will only happen if $f\in I_X\cap \< \bbx\>^t$, establishing the claim. 
We are now left to produce a nonzero divisor on $S/\left\<I_X\cap \<\bbx\>^t,y_0\right\>$. 

To find such a nonzero divisor, note that since $t\geq |\pi_1(X)|-1$, \cref{prop:primaryDecompPPnxPPm} provides the primary decomposition
    \[
    \< I_X \cap \< \bbx \>^t, y_0 \> 
    = 
    \bigcap_{P_k\in \pi_1(X)} \< I_{X_k}, y_0 \> 
    \cap 
    \< \< \bbx \>^t, y_0 \>.
    \]
Now from \eqref{eq:IXk}, 
\begin{align*}
\< I_{X_k}, y_0\> 
&= 
I_{P_k} + \bigcap_{j=1}^{r_k} I_{Q_{k,j}} + \< y_0 \>, 
\end{align*}
so the zero divisors of $S/\< I_X\cap\< \bbx\>^t,y_0\>$ are the elements belonging to the associated primes of $S/\< I_X\cap \< \bbx \>^t,y_0\>$, namely 
\begin{align}
\label{eq:ZDs}
\bigcup_{P_k\in \pi_1(X)} \sqrt{\< I_{X_k}, y_0 \>} 
\cup 
\sqrt{\< \< \bbx \>^t, y_0 \>}
=
\bigcup_{P_k\in \pi_1(X)}\< I_{P_k},G_1,G_2, \ldots, G_s,y_0  \> \cup \< \bbx,y_0 \>, 
\end{align}
where each $G_i$ is a polynomial in only the $\bby$-variables.    
Thus, let $L$ be any element of degree $(1,0)$ in $S$ such that $L$ does not vanish at any point in $\pi_1(X)$. 
In this way, $L+y_1$ does not belong to the set in \eqref{eq:ZDs}, so $L+y_1$ is a nonzero divisor on $S/\< I_X\cap \< \bbx \>^t,y_0\>$. 
Therefore, $S/\< I_X\cap \< \bbx \>^t\>$ has depth at least $2$, so its minimal free resolution has length (at most) $n+m$, as desired.
\end{proof}

\begin{remark}
Note that the bound shown in \cref{thm:intersectionPPnxPPm} is not necessarily a sharp one. 
For points with generic Hilbert function, we can improve the bound given in \cref{thm:intersectionPPnxPPm}. 
In this case, the minimal free resolution of $I_X \cap \< \bbx\>^t$ is a virtual resolution of $I_X$ of length $m+n$ whenever $t \ge \min\{r \mid T_{r,n} \ge |X|\}$. To see this, apply \cref{cor:VSgp} and then follow the proofs of \cref{prop:primaryDecompPPnxPPm} and \cref{thm:intersectionPPnxPPm}.
\end{remark}

%%%%%%%%%%%%%%%%%%%%%%%%%%%%%%%%%%%%%%%%%%%%%%%
\section{Difference Hilbert matrices and Betti numbers} 
%%%%%%%%%%%%%%%%%%%%%%%%%%%%%%%%%%%%%%%%%%%%%%%
\label{sec:MatrixandBetti}

In this section, we prove \cref{prop:CRgeneral}, which relates a certain difference matrix of a Hilbert matrix to the Betti numbers of the corresponding free resolution. 

Recall that $S$ is the Cox ring of $\Pnm$ and $M$ a finitely generated bigraded $S$-module with minimal bigraded free resolution $F_\bullet$ with Betti numbers $\beta_{k,(i,j)}$ and for all $(p,q)\in\ZZ^2$, 
$B_{p,q} = \sum_{k=0}^{\infty} (-1)^k\beta_{k,(p,q)}$. 
Recall $\Delta^C$ and $\Delta^R$ from \cref{def:column-row-difference-operators}, and let 
\[
T_{i,n} := \binom{i+n}{n}
\]
denote the number of monomials of degree $i$ in $\Bbbk[x_0,x_1,\dots,x_n]$, 
so $T_{i,n}T_{j,m}=H_S(i,j)$.
\begin{lem}
\label{lem:Hij}
There is an equality 
$H_M(i,j)
=\sum_{(p,q)\leq (i,j)}T_{i-p,n}T_{j-q,m}B_{p,q}$.
\end{lem}
\begin{proof}
For all $k\geq 0$, $H_S(i,j) = T_{i,n}T_{j,m}$ and 
$F_k = \bigoplus_{p,q\geq 0}S(-p,-q)^{\beta_{k,(p,q)}}$, so 
\begin{align}\label{eq:HFk}
H_{F_k}(i,j) 
= \dim_{\Bbbk}(F_k)_{i,j} 
&=\sum_{(p,q)\leq(i,j)}\dim_{\Bbbk}(S(-p,-q)_{i,j})^{\beta_{k,(p,q)}} \nonumber
\\
&=\sum_{(p,q)\leq(i,j)}\dim_{\Bbbk}(S_{i-p,j-q})^{\beta_{k,(p,q)}} \nonumber
\\
&=\sum_{(p,q)\leq(i,j)}T_{i-p,n}T_{j-q,m}\beta_{k,(p,q)}. 
\end{align}
By definition of $B_{p,q}$, for any $(p,q)\leq (i,j)$ and $k$ sufficiently large, $\beta_{k,(p,q)} = 0$ and thus $H_{F_k}(i,j) = 0$. 
Since $0\gets M\gets F_\bullet$ is an exact sequence in the category of graded $S$-modules, 
\begin{equation*}
H_{M}(i,j) 
= \sum_{k=0}^{\infty}(-1)^k H_{F_k}(i,j).
\end{equation*}
Substituting \eqref{eq:HFk} into the right hand side now yields the desired result.
\end{proof}

\begin{lem}
\label{lem:repeatDeltaCR}
Two identities hold for $\Delta^C$ and $\Delta^R$ applied to the Hilbert matrix of $M$: 
\begin{align*}
(\Delta^C H_M)(i,j)
&=\sum_{(p,q)\leq (i,j)}T_{i-p,n-1}T_{j-q,m}B_{p,q}
\\ 
\text{and}\quad 
(\Delta^R H_M)(i,j)
&=\sum_{(p,q)\leq (i,j)}T_{i-p,n}T_{j-q,m-1}B_{p,q}.
\end{align*}
\end{lem}
\begin{proof}
We prove the identity for $\Delta^C$; the proof for $\Delta^R$ is analogous. 
By \cref{lem:Hij}, 
\begin{align*}
H_M(i,j) 
&=\sum_{(p,q)\leq (i,j)}T_{i-p,n}T_{j-q,m}B_{p,q}. 
\end{align*}
Now separate the sum into two parts, 
noting that when $p = i$, $T_{i-p,n} = T_{0,n} = 1$, so that 
\begin{align*}
H_M(i,j)
&=\sum_{q\leq j}T_{j-q,m}B_{i,q} +\sum_{(p,q)\leq (i-1,j)}T_{i-p,n}T_{j-q,m}B_{p,q}. 
\end{align*}
By \cref{lem:Hij} applied to $(i-1,j)$,
\[
H_M(i-1,j) = \sum_{(p,q)\leq (i-1,j)}T_{i-1-p,n}T_{j-q,m}B_{p,q}. 
\]
Now by definition of $\Delta^C$, 
\begin{align}
\label{eq:DeltaC-iterated}
(\Delta^C H_M)(i,j) 
&= H_M(i,j) - H_M(i-1,j)
\\ 
&= \sum_{q\leq j} T_{j-q,m} B_{i,q} + \sum_{(p,q) \leq (i-1,j)}( T_{i-p,n}-T_{i-1-p,n} ) T_{j-q,m} B_{p,q}. 
\nonumber
\end{align}
Using the binomial identities 
$T_{i-p,n} - T_{i-1-p,n} = T_{i-p,n-1}$, 
\eqref{eq:DeltaC-iterated} simplifies to
\begin{align*}
(\Delta^C H_M)(i,j) 
&= \sum_{q\leq j}T_{j-q,m}B_{i,q} + \sum_{(p,q)\leq (i-1,j)}T_{i-p,n-1}T_{j-q,m}B_{p,q}.
\end{align*}
Recombining the last two sums yields 
\[
(\Delta^C H_M)(i,j) = \sum_{(p,q)\leq (i,j)}T_{i-p,n-1}T_{j-q,m}B_{p,q}.
\qedhere
\]
\end{proof}

\cref{prop:CRgeneral} now follows by applying each of $\Delta^C$ and $\Delta^R$ repeatedly.

\begin{proof}[Proof of \cref{prop:CRgeneral}]
Combining Lemmas~\ref{lem:Hij} and~\ref{lem:repeatDeltaCR}, 
\begin{align*}
G_M(i,j) &\defeq 
\big((\Delta^C)^{n}(\Delta^R)^{m}H_M\big)(i,j)
\\
&\;= \sum_{(p,q)\leq (i,j)}T_{i-p,0}T_{j-q,0}B_{p,q}
\\
&\;= \sum_{(p,q)\leq (i,j)}B_{p,q}.
\end{align*}
Now applying $\Delta^C$, 
\begin{align*}
(\Delta^C G_M)(i,j) 
&= G_M(i,j) - G_M(i-1,j) 
\\
&= \sum_{(p,q)\leq (i,j)}B_{p,q} - \sum_{(p,q)\leq (i-1,j)}B_{p,q} 
\,\,=\,\, \sum_{q\leq j} B_{i,q}, 
\intertext{so applying $\Delta^R$ yields the desired result: }
(\Delta^R\Delta^C G_M)(i,j) 
&= \sum_{q\leq j} B_{i,q} - \sum_{q\leq j-1} B_{i,q} 
\,=\, B_{i,j}.
\qedhere
\end{align*}
\end{proof}

%%%%%%%%%%%%%%%%%%%%%%%%%%%%%%%%%%%%%%%%%%%%%%%
\section{Sufficiently general points in  \texorpdfstring{$\PP^1 \times \PP^2$}{PP1 x PP2}}
%%%%%%%%%%%%%%%%%%%%%%%%%%%%%%%%%%%%%%%%%%%%%%%
\label{sec:betti-points-mrc}

We now discuss the notion of sufficiently general points, as per the setting of \cref{thm:virtual-of-a-pair}. 
We begin with the notion of a generic Hilbert function for $S/I_X$, sketched for us by Gregory G. Smith. 
While this result holds over any smooth toric variety, for continuity of notation, we state the result only for $\Pnm$. 

\begin{prop}
\label{prop:generic}
For every $N\geq 1$ and all $n,m\geq 1$, there exists a dense open subset $U\subseteq (\Pnm)^N$, such that for any $X=\{P_1,P_2,\dots,P_N\}\subseteq \Pnm$, if $(P_1,P_2,\dots,P_N)\in U$, then 
\[
H_X(i,j)=\min\{|X|,T_{i,n}T_{j,m}\} 
\quad\text{for all}\quad 
i,j\geq 0,
\]
where $T_{i,n} \defeq \binom{i+n}{n}$ 
is equal to the number of monomials of degree $i$ in $\Bbbk[x_0,x_1,\dots,x_n]$. 
\end{prop}
\begin{proof}
A natural parameter space for $N$ distinct points in $\Pnm$ is simply the open subset of $(\Pnm)^N$, the $N$-fold product of $\Pnm$, obtained by removing the large diagonals, i.e., subsets where two points are equal.  
Note that the Hilbert function of the Cox ring $S$ of $\Pnm$ at $(i,j)\in\NN^2$ is $H_S(i,j) = T_{i,n}T_{j,m}$. 

If the Hilbert function of a set $Y$ of $N$ distinct points disagrees with $H_S(i,j)$ for some $(i,j)\in\NN^2$, then it must be strictly smaller. 
It follows that matrix with $N$ rows and $H_S(i,j)$ columns, corresponding to evaluating the $N$ points in $Y$ at the monomial basis for the Cox ring $S$ in degree $(i,j)$ has less than maximal rank.  
Thus the corresponding maximal minors of the matrix, interpreted as polynomial equations in the Cox ring of the parameter space $(\Pnm)^N$, determine a closed subset that contains $Y$, as desired. 
\end{proof}

\begin{example}
\label{ex:HXforPP1xPP2}
If $X\subseteq \PP^1\times\PP^2$ is a finite set of sufficiently general points with $|X|=N\geq 12$,  
\[
N = 6q+r=3q'+r' 
\quad\text{with}\quad 
0\leq r\leq 5 
\quad\text{and}\quad 
0\leq r'\leq 2,
\]
$c_j=\max\{c \mid c T_{j,2} < N\}$, and  
$d=\min\{j \mid T_{j,2} \ge |X|\}$, 
then the Hilbert matrix of $S/I_X$ is 
\[
\small
H_X = \bordermatrix{
~ & 0 & 1 & 2 & \cdots & j & \cdots & d & d+1 & \cdots \cr
0 & 1 & 3 & 6 & \cdots & T_{j,2} & \cdots & N & N & \cdots \cr
1 & 2 & 6 & 12 & \cdots & 2T_{j,2} & \cdots & N & N & \cdots \cr
\vdots & \vdots & \vdots & \vdots & & \vdots & & \vdots & \vdots & \cr
c_j-1 & c_j & 3c_j & 6c_j & \cdots & c_jT_{j,2} & \cdots & N & N & \cdots \cr
c_j & c_j+1 & 3(c_j+1) & 6(c_j+1) & \cdots & N & \cdots & N & N & \cdots \cr
\vdots & \vdots & \vdots & \vdots & & \vdots & & \vdots & \vdots & \cr
q-1 & q & 3q & 6q & \cdots & N & \cdots & N & N & \cdots \cr
q & q+1 & 3q+3 & N & \cdots & N & \cdots & N & N & \cdots \cr
\vdots & \vdots & \vdots & \vdots & & \vdots & & \vdots & \vdots & \cr
q'-1 & q' & 3q' & N & \cdots & N & \cdots & N & N & \cdots \cr
q' & q'+1  & N & N & \cdots & N & \cdots & N & N & \cdots \cr
\vdots & \vdots & \vdots & \vdots & & \vdots & & \vdots & \vdots & \cr
N-1 & N & N & N & \cdots & N & \cdots & N & N & \cdots \cr
N & N & N & N & \cdots & N & \cdots & N & N & \cdots \cr
\vdots & \vdots & \vdots & \vdots & & \vdots & & \vdots & \vdots & \ddots\cr
}. 
\normalsize
\qedhere
\]
\end{example}

Moving towards \cref{thm:virtual-of-a-pair}, 
we now turn our attention to the specific setting of $\PP^1 \times \PP^2$, for which we need an additional assumption on set of points that restricts their first Betti numbers. 

\begin{notation}
    \label{not:difference}
    Given a matrix $H$ with indices $(i,j)\in \ZZ^2$, set 
\begin{equation}
\label{eq:DH}
DH \defeq 
(\Delta^C)^2 (\Delta^R)^3 H.
\end{equation}
\end{notation}

For a finite set of points $X$ in $\PP^1\times\PP^2$, 
\cref{prop:CRgeneral} implies that the Betti numbers of $S/I_X$ satisfy 
\begin{equation}
\label{eq:CR-P1P2}
DH_X(i,j)=-\beta_{1,(i,j)}+\beta_{2,(i,j)}-\beta_{3,(i,j)}+\beta_{4,(i,j)}.
\end{equation}
for all $(i,j)>(0,0)$.
Ideally, we would like to say that under certain assumptions (for example, when $X$ has generic Hilbert matrix), the matrix $DH_X$ completely determines all Betti numbers of $S/I_X$. This amounts to saying that a certain version of the Minimal Resolution Conjecture (MRC) holds for sets of points in $\PP^1\times\PP^2$. 

Originally formulated by Lorenzini  in \cite{lorenzini} for points in $\PP^n$, the MRC states that for a set of points in sufficiently general position, Betti numbers cannot overlap, meaning that $\beta_{i,j}\beta_{i+1,j}=0$ for all $0\leq i\leq n$ and $j\in\ZZ$. 
Later, \mustata \; generalized the MRC to sets of points in arbitrary projective varieties~\cite{mustata}. 
A more detailed discussion of the history and current status of the MRC can be found in the introduction of \cite{booms-peot}. 

In \cite{giuffrida-96}, it is shown that the MRC holds for all sufficiently general sets of points lying on a smooth quadric in $\PP^3$. 
The techniques of \cite{giuffrida-96} are highly specific to the geometry of $\PP^3$ and therefore difficult to generalize to higher dimensions or multiprojective spaces. 
The recent studies \cites{HNV-points,booms-peot} of virtual resolutions for points in $\PP^1\times\PP^1$ show a weakened partial version of the MRC, which provides enough control to obtain their results.  

For the purpose of our study, we now state a weakened version of the MRC for points in $\PP^1\times\PP^2$ that parallels \cite{giuffrida-96}. 
Here we only require that the Betti numbers $\beta_{1,(i,j)}$ are entirely predicted by $DH_X$. 
This assumption provides enough control over Betti numbers to prove \cref{thm:virtual-of-a-pair}.

\begin{conjecture}[Weakened Minimal Resolution Conjecture] 
\label{conj:mrc}
Let $N\geq 2$ be an integer. 
There exists an dense open subset $U\subseteq(\PP^1\times\PP^2)^N$ such that for every $(P_1,P_2,\dots,P_N)\in U$, the set of points $X=\{P_1,\dots,P_N\}$ satisfies: 
\begin{enumerate}
\item $X$ has a generic Hilbert matrix as in \cref{def:suffGeneral}, and 
\item For every fixed $(i,j)>(0,0)$, 
the bigraded Betti numbers of $S/I_X$ are such that $\beta_{1,(i,j)}>0$ if and only if 
\begin{align*}
\label{eq:mrc-condition}
DH_X(i,j)<0 \textrm{ and } DH_X(i',j')\leq 0 \textrm{ for all } (i',j')\leq (i,j) \textrm{ with } (i',j') \ne (0,0), 
\end{align*}
with $DH$ as in \cref{not:difference}, and whenever the above condition holds, $\beta_{1,(i,j)} = -DH_X(i,j)$. 
\end{enumerate}
\end{conjecture}

\begin{example}
\label{ex:check-mrc-PP1xPP2}
For each $2\leq N\leq 100$, we verified in Macaulay2 that \cref{conj:mrc} holds for $50$ random sets of $N$ points in $\PP^1\times\PP^2$ with generic Hilbert function \cite{M2}. 
\end{example}

%%%%%%%%%%%%%%%%%%%%%%%%%%%%%%%%%%%%%%%%%%%%%%%
\section{A short virtual of a pair resolution in \texorpdfstring{$\PP^1 \times \PP^2$}{PP1 x PP2}}
%%%%%%%%%%%%%%%%%%%%%%%%%%%%%%%%%%%%%%%%%%%%%%%
\label{sec:VResPair}

In this section, we prove  \cref{thm:virtual-of-a-pair}. 
To begin, 
The virtual of a pair construction is one of the methods given by Berkesch, Erman, and Smith to construct virtual resolutions, and it relies on Maclagan and Smith's notion of multigraded regularity. 
We do not define this invariant in general, but instead provide a simplified characterization for our specific setting. 

\begin{prop}[See \cite{MS}*{Proposition 6.7}]
\label{prop:MS-HFvalue}
Let $X$ be a finite set of points in $\Pnm$. Then the \emph{multigraded regularity} of $S/I_X$ is 
\[
\reg(S/I_X) = \left\{ \bbd\in\ZZ^2\mid H_X(\bbd) = |X| \right\}. 
\]
\end{prop}

\begin{example}
    Recall that for points in sufficiently general position in $\PP^1 \times \PP^2$, the Hilbert matrix is given by
    \[
    H_X(i,j)=\min\{|X|,T_{i,1}T_{j,2}\} 
    \quad\text{for all}\quad 
    i,j\geq 0,
    \]
    Therefore if $|X|=N$, then 
    \[
    H_X(N-1,0)=\min\{N,T_{N-1,1}T_{0,2}\}=\min\{N, N\}=N, 
    \]
    so $(N-1,0) \in \reg(S/I_X)$ for all finite sets of $N$ points $X \subseteq \PP^1 \times \PP^2$.
\end{example}

\begin{thm}[See \cite{virtual-original}*{Theorem 1.3}]
\label{thm:trimming}
Let $S$ be the Cox ring of $\Pnm$, let $M$ be a finitely generated bigraded $B$-saturated $S$-module, and let $\bbd\in \ZZ^2$ be such that $\bbd\in\reg(M)$. 
If $G_\bullet$ is the free subcomplex of a minimal free resolution $F_\bullet$ of $M$ consisting of all summands of $F_\bullet$ generated in degree at most $\bbd+(n,m)$, then $G_\bullet$ is a virtual resolution of $M$.
\end{thm}

As in \cite{virtual-original}, we call the virtual resolution in \cref{thm:trimming} the \emph{virtual resolution of the pair} $(M,\bbd)$. 
With this in hand, we are now prepared to state the main result of this section, a more explicit version of \cref{thm:virtual-of-a-pair} when $|X| \ge 12$. The case for smaller sets of points is handled in \cref{rmk:special-cases}. 

\begin{thm}[See \cref{thm:virtual-of-a-pair}]
\label{thm:virtual-of-pair-stronger}
Let $X\subseteq \PP^1\times\PP^2$ be a finite set of sufficiently general points with $N=|X|\geq 12$. 
If \Cref{conj:mrc} holds, 
then the virtual of a pair construction for $\bbd=(N-1,0)$ yields a virtual resolution of $S/I_X$ of the form 
\small
\[
S 
\gets 
\begin{gathered}S(-q,-2)^{6-r} \\[-5pt] \oplus \\[-5pt] S(-q-1,-2)^r \\[-5pt] \oplus \\[-5pt] S(-q',-1)^{3-r'} \\[-5pt] \oplus \\[-5pt] S(-q'-1,-1)^{r'} \\[-5pt] \oplus \\[-5pt] S(-N,0)\end{gathered} 
\gets
\begin{gathered}S(-q',-2)^{9-3r'} \\[-5pt] \oplus \\[-5pt] S(-q'-1,-2)^{3r'} \\[-5pt] \oplus \\[-5pt] S(-N,-1)^3\end{gathered} 
\gets
S(-N,-2)^3 
\gets 
0,
\]
\normalsize
where $N=6q+r=3q'+r'$ with $0\leq r\leq 5$ and $0\leq r'\leq 2$. 
When $r$ (respectively, $r'$) is zero, the term $S(-q-1,-2)^r$ (respectively, $S(-q'-1,-1)^{r'}$ and $S(-q'-1,-2)^{3r'}$) do not appear in the virtual resolution.
\end{thm}

\begin{remark}
\label{rmk:special-cases}
In the case that $2\leq |X|\leq 11$, if we assume \cref{conj:mrc}, then it is still true that $S/I_X$ has a virtual resolution of length $3$ obtained by taking $\bbd=(|X|-1,0)$ in \cref{thm:trimming}, and the proofs are analogous to \cref{thm:virtual-of-pair-stronger}, but the resulting virtual resolution has a slightly different form. 
These cases are listed in \cref{sec:cases-2-to-11}, and with \cref{thm:virtual-of-pair-stronger}, complete the proof of \cref{thm:virtual-of-a-pair}.
\end{remark}

We now state two lemmas for the proof of \cref{thm:virtual-of-pair-stronger}. 

\begin{lem}
\label{lem:submatrix}
Let $X\subseteq \PP^1\times \PP^2$ be a collection of points with generic Hilbert matrix as in \cref{def:suffGeneral} with $N = |X|\geq 12$. 
If 
\[
\small 
N = 6q+r=3q'+r' 
\quad\text{with}\quad 
0\leq r\leq 5 
\quad\text{and}\quad 
0\leq r'\leq 2,
\]
\normalsize
then using \cref{not:difference}, 
\vspace*{-4mm}
\small
\begin{equation*}
\label{eq:DH-generic}
DH_X=
\begin{blockarray}{ccccc}
\begin{block}{c(cccc)} 
0 & 1 & 0 & 0 & \cdots\\
1 & 0 & 0 & 0 & \cdots \\
\vdots & \vdots & \vdots & \vdots & \\
q-1 & 0 & 0 & 0 & \cdots\\
q &0 & 0 & r-6& \cdots\\
q+1& 0 & 0 & -r& \cdots\\
q+2 & 0 & 0 & 0& \cdots\\
\vdots & \vdots & \vdots & \vdots & \\
q'-1 & 0 & 0 & 0& \cdots\\
q' & 0 & r'-3 & 9-3r'& \cdots\\
q'+1 & 0 & -r' & 3r'& \cdots\\
q'+2 & 0 & 0 & 0& \cdots\\
\vdots & \vdots & \vdots & \vdots & \\
N-1 & 0 & 0 & 0 & \cdots\\
N & -1 & 3 & -3 & \cdots \\
N+1 & 0& 0 & 0 & \cdots \\
&  \vdots & \vdots & \vdots & \ddots\\
\end{block}
\end{blockarray}\ .
\end{equation*}
\normalsize
\end{lem}
\begin{proof}
Recall from \cref{def:column-row-difference-operators} that $\Delta^C$ and $\Delta^R$ commute, and set $\Delta\defeq \Delta^C\Delta^R=\Delta^R\Delta^C$.
In light of \cref{ex:HXforPP1xPP2}, straightforward computation yields
\[
\small
H_X =
\begin{blockarray}{ccccc}
\begin{block}{c(cccc)} 
0  & 1 & 3 & 6 & \cdots\\ 
1  &2 & 6 & 12 & \cdots\\ 
\vdots  &\vdots & \vdots & \vdots & \\ 
q-1  &q & 3q & 6q & \cdots\\ 
q    &q+1 & 3q+3 & N & \cdots\\ 
q+1  &q+2 & 3q+6 & N & \cdots\\ 
\vdots  &\vdots & \vdots & \vdots &\\ 
q'-1  &q' & 3q' & N & \cdots\\ 
q'    &q'+1 & N & N & \cdots\\ 
q'+1  &q'+2 & N & N & \cdots\\ 
\vdots  &\vdots & \vdots & \vdots &\\ 
N-1  &N & N & N & \cdots\\ 
N    &N & N & N & \cdots \\
&\vdots & \vdots & \vdots & \ddots \\
\end{block}
\end{blockarray}
\quad\ \text{\normalsize and}\quad 
\Delta H_X = 
\begin{blockarray}{ccccc}
\begin{block}{c(cccc)} 
0  &1 & 2 & 3 & \cdots\\ 
1  &1 & 2 & 3 & \cdots\\ 
\vdots  &\vdots & \vdots & \vdots &\\ 
q-1  &1 & 2 & 3 & \cdots\\ 
q    &1 & 2 & r-3 & \cdots\\ 
q+1  &1 & 2 & -3 & \cdots\\ 
\vdots  &\vdots & \vdots & \vdots &\\ 
q'-1  &1 & 2 & -3 & \cdots\\ 
q'    &1 & r'-1 & -r' & \cdots\\ 
q'+1  &1 & -1 & 0 & \cdots\\ 
\vdots  &\vdots & \vdots & \vdots &\\ 
N-1  &1 & -1 & 0 & \cdots\\ 
N    &0 & 0 & 0 & \cdots\\
&\vdots & \vdots & \vdots & \ddots \\
\end{block}
\end{blockarray}\ .
\normalsize
\vspace*{-2mm}
\]
The result follows from direct computation. 
\end{proof}

\begin{lem}
\label{lem:beta2-plus}
Let $X$ be a set of $N\geq 12$ points in sufficiently general position in $\PP^1 \times \PP^2$, and let $F_\bullet$ be the minimal free resolution of $S/I_X$ with Betti numbers $\beta_{k,(i,j)}$.
If $DH_X(i,j)$ is the first positive entry in the $i$th row of $DH_X$ (excluding the $0$th and $1$st row), 
then $\beta_{2,(i,j)} = DH_X(i,j)$, i.e., that positive entry corresponds exactly to the number of first syzygies of $S/I_X$ of degree $(i,j)$. 
Furthermore, $\beta_{2,(i,j')} = 0$ for all $j' < j$, 
i.e., there are no first syzygies of smaller degree coming from that row. 
\end{lem}
\begin{proof}
This argument follows the approach of \cite{booms-peot}*{Lemma~3.2}. 
Let $F_\bullet$ be the minimal free resolution of $S/I_X$. For $i\geq 0$, let $C^{\leq i}$ denote the subcomplex of $F_\bullet$ with summands generated in degree $\ba = (a_1,a_2)$ such that $a_1\leq i$. 
These complexes provide a filtration of $F_\bullet$, so for $i\geq 1$, let $C^i$ denote the cokernel of the natural inclusion $C^{\leq i-1}\hookrightarrow C^{\leq i}$. 

When $K\gg N$, we claim that $C^{\leq i}$ is a virtual resolution of a pair for $\left(S/I_X,(i-1,K-2)\right)$. 
To see this, note that $(0,N-2)\in \reg(S/I_X)$ by \cref{prop:MS-HFvalue}. 
Thus $(i-1,K-2)\in\reg(S/I_X)$ for all $i\geq 1, K\gg N$. 
By choosing $K$ to be the largest degree in the second coordinate of a generator of any summand in $C^{\leq i}$, 
$C^{\leq i}$ is exactly the virtual resolution of the pair $(S/I_X,(i-1,K-2))$. 
Now $0\to C^{\leq i-1}\to C^{\leq i}\to C^i\to 0$ is a short exact sequence of complexes, and $C^{\leq i-1}$ and $C^{\leq i}$ are both virtual resolutions of $S/I_X$ by \cref{thm:trimming}, so $C^i$ must have irrelevant homology. 
Further, since the first coordinates of the degrees of all generators in $C^i$ are fixed at $i$, it is actually a complex over $\Bbbk[y_0,y_1,y_2]$ with homology modules annihilated by some power of $\<y_0,y_1,y_2\>$. The Acyclicity Lemma \cite{eisenbud}*{Lemma~20.11} thus ensures that $H_k(C^i)=0$ for all $k>1$. 

Now if $DH_X(i,j)$ is the first positive entry with $i\geq 2$ in the $i$th row of $DH_X$, then 
\[
C^i\colon\quad 0\gets C^i_1\gets C^i_2\gets C^i_3\gets C^i_4\gets 0
\]
is a minimal free resolution over $\Bbbk[y_0,y_1,y_2]$, so there are no units in its maps and the minimal generating degrees of summands must increase with homological degree. 

By \cref{conj:mrc}, it follows that the Betti numbers $\beta_{1,(i,j')}$ come precisely from the first negative entries of the $i$th row of $DH_X$, where $j'<j$. 
By inspection from \cref{lem:submatrix}, 
$0>DH_X(i,j-1) = \beta_{1,(i,j-1)}$ and thus $\beta_{1,(i,j)}=0$. 
Further, since the minimal degree of generators must increase in $C^i$, 
$\beta_{3,(i,j)}=\beta_{4,(i,j)}=0$ and 
$DH_X(i,j) = \beta_{2,(i,j)}$, as desired. 
\end{proof}

We are now prepared to prove the main theorem of this section. 

\begin{proof}[Proof of \cref{thm:virtual-of-pair-stronger}]
Since $X$ has the generic Hilbert matrix in \cref{def:suffGeneral}, $H_X(N-1,0) = |X|=N$, so $(N-1,0)\in \reg(S/I_X)$ by \cite{MS}*{Proposition 6.7}. 
Thus \cref{thm:trimming} produces a virtual resolution of $S/I_X$ that is the subcomplex of the minimal free resolution of $S/I_X$ consisting of all summands generated in degrees at most $(N,2)$. 

We now examine the minimal free resolution of $S/I_X$. 
By \cref{conj:mrc}, 
the first Betti numbers can be read from the first negative entries in the rows of $DH_X$, 
and these occurs at positions $(q,2), (q+1,2), (q',1), (q'+1,1)$, and $(N, 0)$ by \cref{lem:submatrix}. 
Since graded Betti numbers in a minimal free resolution must increase in degree due to minimality and by \cref{conj:mrc}, it also follows that $\beta_{k,(i,j)}=0$ for $k>1$ for the $(i,j)$ for which $\beta_{1,(i,j)}\neq 0$. 

Next, by \cref{lem:beta2-plus}, 
the second Betti numbers can be read from the first positive entries occurring after a negative entry along a row of $DH_X$. 
By \cref{lem:submatrix}, these occur in positions $(q',2), (q'+1,2)$, and $(N,1)$, with values which match the statement of the theorem. 
Since graded Betti numbers in a minimal free resolution must increase in degree due to minimality, it also follows that $\beta_{k,(i,j)}=0$ for $k>2$ for the $(i,j)$ for which $\beta_{2,(i,j)}\neq 0$. 

Third Betti numbers are now only possible in degrees $(q'+2,2), (q'+3,2),\dots, (N,2)$. However, minimality of the original complex disallows both $\beta_{3,(i,j)}$ and $\beta_{4,(i,j)}$ to be nonzero. 
Thus, only $\beta_{3,(N,2)}$ can be nonzero and $\beta_{4,(N,2)}=0$, so by \cref{prop:CRgeneral}, $\beta_{3,(N,2)}=3$. 
Having exhausted the possible degrees for nonzero Betti numbers, the proof is now complete. 
\end{proof}

It is worth noting that the virtual of a pair construction does not always yield a short virtual resolution, even when the $\bbd\in\reg(S/I_X)$ used is as small as possible. 
In \cite{booms-peot}, Booms-Peot showed that in $\PP^1 \times \PP^1$, under certain conditions virtual of a pair yields a virtual resolution of length $3$. 
The same issue persists in $\PP^1 \times \PP^2$, as shown below. 

\begin{example}
For $X \subseteq \PP^1 \times \PP^2$ with $|X|=31$ in sufficiently general position, the minimal free resolution has total Betti numbers 
\[
S\gets 
S^{34} \gets
S^{66} \gets 
S^{39} \gets 
S^{6}
\gets 0.
\]
Performing virtual of a pair at $(2,4)\in\reg(S/I_X)$ yields the complex 
\[
\small 
0 \gets S/I_X
\gets S
\gets \begin{gathered}S(-3,-3)^{9} \\[-5pt] \oplus \\[-5pt] S(-2,-4)^{14} \\[-5pt] \oplus \\[-5pt] S(-1,-5)^{11}\end{gathered}
\gets
\begin{gathered}S(-3,-4)^{26} \\[-5pt] \oplus \\[-5pt] S(-2,-5)^{32} \\[-5pt] \oplus \\[-5pt] S(-1,-6)^8\end{gathered}
\gets \begin{gathered}S(-3,-5)^{24} \\[-5pt] \oplus \\[-5pt] S(-2,-6)^{15}\end{gathered}
\gets S(-3,-6)^6
\gets 0.
\normalsize
\qedhere
\]
\end{example}

%%%%%%%%%%%%%%%%%%%%%%%%%%%%%%%%%%%%%%%%%%%%%%%%%%%
\appendix
\section{Short virtual of a pair resolutions for small numbers of sufficiently general points}
\label{sec:cases-2-to-11}
%%%%%%%%%%%%%%%%%%%%%%%%%%%%%%%%%%%%%%%%%%%%%%%%%%%
Listed below are the virtual resolutions obtained from performing virtual of a pair at $(|X|-1,0)$ on the minimal free resolutions of $S/I_X$ when $X \subseteq \PP^1 \times \PP^2$ is a set of sufficiently general points with $2 \le |X| \le 11$. 
The proofs here closely follow that of \cref{thm:virtual-of-pair-stronger} and are thus omitted. The main difference is that for a smaller number of points, the difference matrix $DH_X$ is considerably more crowded. 
\vspace*{-2mm}
\begin{multicols}{2}
\begin{itemize}
\item $|X|=2$:
\tiny
\vspace*{-2mm}
\[
S
\gets \begin{gathered}S(0,-2)\\[-5pt] \oplus \\[-5pt] S(0,-1) \\[-5pt] \oplus \\[-5pt] S(-1,-1)^2 \\[-5pt] \oplus \\[-5pt] S(-2,0)\end{gathered}
\gets \begin{gathered}S(-1,-2)^4\\[-5pt] \oplus \\[-5pt] S(-2,-1)^3\end{gathered}
\gets S(-2,-2)^3.
\]
\vspace*{-3mm}

\normalsize
\item $|X|=3$: 
\tiny
\vspace*{-2mm}
\[
S
\gets \begin{gathered}S(0,-2)^3\\[-5pt] \oplus \\[-5pt] S(-1,-1)^3 \\[-5pt] \oplus \\[-5pt] S(-3,0) \end{gathered}
\gets \begin{gathered}S(-1,-2)^6\\[-5pt] \oplus \\[-5pt] S(-3,-1)^3\end{gathered}
\gets S(-3,-2)^3.
\]
\normalsize

\item $|X|=4$: 
\tiny
\vspace*{-2mm}
\[
S
\gets \begin{gathered}S(0,-2)^2\\[-5pt] \oplus \\[-5pt] S(-1,-1)^2 \\[-5pt] \oplus \\[-5pt] S(-2,-1) \\[-5pt] \oplus \\[-5pt] S(-4,0)\end{gathered}
\gets \begin{gathered}S(-1,-2)^2\\[-5pt] \oplus \\[-5pt] S(-2,-2)^3\\[-5pt] \oplus \\[-5pt] S(-4,-1)^3\end{gathered}
\gets S(-4,-2)^3.
\]
\normalsize

\item $|X|=5$:
 \tiny
\vspace*{-2mm}
\[
S
\gets \begin{gathered}S(0,-2)\\[-5pt] \oplus \\[-5pt] 
S(-1,-2)^2 \\[-5pt] \oplus \\[-5pt] 
S(-1,-1) \\[-5pt] \oplus \\[-5pt] 
S(-2,-1)^2 \\[-5pt] \oplus \\[-5pt] S(-5,0)\end{gathered}
\gets \begin{gathered}S(-2,-2)^6\\[-5pt] \oplus \\[-5pt] S(-5,-1)^3\end{gathered}
\gets S(-5,-2)^3.
\]
\normalsize

\item $|X|=6$:
 \tiny
\vspace*{-2mm}
\[
S
\gets \begin{gathered}S(-1,-2)^6\\[-5pt] \oplus \\[-5pt]
S(-2,-1)^3 \\[-5pt] \oplus \\[-5pt] S(-6,0)\end{gathered}
\gets \begin{gathered}S(-2,-2)^9\\[-5pt] \oplus \\[-5pt] S(-6,-1)^3\end{gathered}
\gets S(-6,-2)^3.
\]
\normalsize

\item $|X|=7$:
 \tiny
\vspace*{-1.5mm}
\[
S
\gets \begin{gathered}S(-1,-2)^5\\[-5pt] \oplus \\[-5pt]
S(-2,-1)^2 \\[-5pt] \oplus \\[-5pt] 
S(-3,-1) \\[-5pt] \oplus \\[-5pt] S(-7,0)\end{gathered}
\gets \begin{gathered}S(-2,-2)^5\\[-5pt] \oplus \\[-5pt] 
S(-3,-2)^3\\[-5pt] \oplus \\[-5pt]
S(-7,-1)^3\end{gathered}
\gets S(-7,-2)^3.
\]
\normalsize

\item $|X|=8$: 
 \tiny
\vspace*{-1mm}
\[
S
\gets \begin{gathered}S(-1,-2)^4\\[-5pt] \oplus \\[-5pt]
S(-2,-1) \\[-5pt] \oplus \\[-5pt] 
S(-3,-1)^2 \\[-5pt] \oplus \\[-5pt] S(-8,0)\end{gathered}
\gets \begin{gathered}S(-2,-2)\\[-5pt] \oplus \\[-5pt] 
S(-3,-2)^6\\[-5pt] \oplus \\[-5pt]
S(-8,-1)^3\end{gathered}
\gets S(-8,-2)^3.
\]
\normalsize

\item $|X|=9$: 
 \tiny
\vspace*{-1mm}
\[
S
\gets \begin{gathered}S(-1,-2)^3\\[-5pt] \oplus \\[-5pt]
S(-2,-2)^3 \\[-5pt] \oplus \\[-5pt] 
S(-3,-1)^3 \\[-5pt] \oplus \\[-5pt] S(-9,0)\end{gathered}
\gets \begin{gathered}
S(-3,-2)^9\\[-5pt] \oplus \\[-5pt]
S(-9,-1)^3\end{gathered}
\gets S(-9,-2)^3.
\]
\normalsize

\item $|X|=10$: 
 \tiny
\vspace*{-1mm}
\[
S
\gets \begin{gathered}S(-1,-2)^2\\[-5pt] \oplus \\[-5pt]
S(-2,-2)^4 \\[-5pt] \oplus \\[-5pt] 
S(-3,-1)^2 \\[-5pt] \oplus \\[-5pt] 
S(-4,-1) \\[-5pt] \oplus \\[-5pt] S(-10,0)\end{gathered}
\gets \begin{gathered}
S(-3,-2)^6\\[-5pt] \oplus \\[-5pt]
S(-4,-2)^3\\[-5pt] \oplus \\[-5pt]
S(-10,-1)^3\end{gathered}
\gets S(-10,-2)^3.
\]
\normalsize

\item $|X|=11$: 
 \tiny
\vspace*{-1mm}
\[
S
\gets \begin{gathered}S(-1,-2)\\[-5pt] \oplus \\[-5pt]
S(-2,-2)^5 \\[-5pt] \oplus \\[-5pt] 
S(-3,-1) \\[-5pt] \oplus \\[-5pt] 
S(-4,-1)^2 \\[-5pt] \oplus \\[-5pt] S(-11,0)\end{gathered}
\gets \begin{gathered}
S(-3,-2)^3\\[-5pt] \oplus \\[-5pt]
S(-4,-2)^6\\[-5pt] \oplus \\[-5pt]
S(-11,-1)^3\end{gathered}
\gets S(-11,-2)^3.
\]
\normalsize
\end{itemize}
\end{multicols}
%%%%%%%%%%%%%%%%%%%%%%%%%%%%%%%%%%%%%%%%%%%%%%%%%%%
%%%%%%%%%%%%%%%%%%%%%%%%%%%%%%%%%%%%%%%%%%%%%%%%%%%
%%%%%%%%%%%%%%%%%%%%%%%%%%%%%%%%%%%%%%%%%%%%%%%%%%%

%%%%%%%%%%%%%%%%%%%%%%%%%%%%%%%%%%%%%%%%%%%%%%%%%%%

\begin{bibdiv}
\begin{biblist}

\bib{virtualJSAG}{article}{
   author={Almousa, Ayah},
   author={Bruce, Juliette},
   author={Loper, Michael},
   author={Sayrafi, Mahrud},
   title={The virtual resolutions package for Macaulay2},
   journal={J. Softw. Algebra Geom.},
   volume={10},
   date={2020},
   number={1},
   pages={51--60},
}

\bib{reggie-thesis}{article}{
    author={Anderson, Reginald},
    title={A Resolution of the Diagonal for Toric Deligne--Mumford Stacks}, 
    journal={\textsf{ arXiv:2303.17497 [math.AG]}},
    pages={52 pages},
}


\bib{virtual-original}{article}{
   author={Berkesch, Christine},
   author={Erman, Daniel},
   author={Smith, Gregory G.},
   title={Virtual resolutions for a product of projective spaces},
   journal={Algebr. Geom.},
   volume={7},
   date={2020},
   number={4},
   pages={460--481},
}

\bib{virtualCM}{article}{
   author={Berkesch, Christine},
   author={Klein, Patricia},
   author={Loper, Michael C.},
   author={Yang, Jay},
   title={Homological and combinatorial aspects of virtually Cohen-Macaulay
   sheaves},
   journal={Trans. London Math. Soc.},
   volume={8},
   date={2021},
   number={1},
   pages={413--434},
}

\bib{booms-peot}{article}{
   author={Booms-Peot, Caitlyn},
   title={Hilbert-Burch virtual resolutions for points in $\PP^1 \times \PP^1$},
   journal={arXiv.AC:2304.04953},
   pages={21 pages},
}

\bib{booms-cobb}{article}{
   author={Booms-Peot, Caitlyn},
   author={Cobb, John},
   title={Virtual criterion for generalized Eagon-Northcott complexes},
   journal={J. Pure Appl. Algebra},
   volume={226},
   date={2022},
   number={12},
   pages={Paper No. 107138, 8},
}

\bib{short-HST}{article}{
    author={Brown, Michael K.},
    author={Erman, Daniel},
    title={Results on virtual resolutions for toric varieties}, 
    journal={\textsf{arXiv:2303.14319 [math.AG]}},
    pages={5 pages},
}


\bib{mahrud-resDiag}{article}{
    author={Brown, Michael K.},
    author={Sayrafi, Mahrud}, 
    title={A short resolution of the diagonal for smooth projective toric varieties of Picard rank 2}, 
    journal={\textsf{arXiv:2208.00562 [math.AG]}}, 
    pages={18 pages},
}

\bib{duarte-seceleanu}{article}{
    author = {Duarte, Eliana}, 
    author = {Seceleanu, Alexandra},
     TITLE = {Implicitization of tensor product surfaces via virtual
              projective resolutions},
   JOURNAL = {Math. Comp.},
%  FJOURNAL = {Mathematics of Computation},
    VOLUME = {89},
      YEAR = {2020},
    NUMBER = {326},
     PAGES = {3023--3056},
}

\bib{eisenbud}{book}{
author={Eisenbud, David},
title={Commutative Algebra: With a View Toward Algebraic Geometry},
year={1995},
publisher={Springer New York}
}

\bib{FH}{article}{
    author={Favero, David}, 
    author={Huang, Jesse}, 
    title={Rouquier dimension is Krull dimension for normal toric varieties}, 
    journal={\textsf{arXiv.2302.09158 [math.AG]}}, 
    pages={9 pages},
}

\bib{reu2018}{article}{
   author={Gao, Jiyang},
   author={Li, Yutong},
   author={Loper, Michael C.},
   author={Mattoo, Amal},
   title={Virtual complete intersections in $\PP^1 \times \PP^1$},
   journal={J. Pure Appl. Algebra},
   volume={225},
   date={2021},
   number={1},
   pages={Paper No. 106473, 15},
}

\bib{giuffrida-92}{article}{
    author = {Giuffrida, S.},
    author = {Maggioni, R.}, 
    author = {Ragusa, A.},
    title = {On the postulation of {$0$}-dimensional subschemes on a smooth quadric},
    journal = {Pacific J. Math.},
    volume = {155},
      year = {1992},
    number = {2},
     PAGES = {251--282},
}

\bib{giuffrida-96}{article}{
    author = {Giuffrida, S.},
    author = {Maggioni, R.}, 
    author = {Ragusa, A.},
     TITLE = {Resolutions of generic points lying on a smooth quadric},
   JOURNAL = {Manuscripta Math.},
    VOLUME = {91},
      YEAR = {1996},
    NUMBER = {4},
     PAGES = {421--444},
}

\bib{HHL}{article}{
    author={Hanlon, Andrew}, 
    author={Hicks, Jeff}, 
    author={Lazarev, Oleg}, 
    title={Resolutions of toric subvarieties by line bundles and applications}, 
    journal={\textsf{arxiv:2303.03763 [math.AG]}}, 
    pages={63 pages}, 
}

\bib{HNV-points}{article}{
   author={Harada, Megumi},
   author={Nowroozi, Maryam},
   author={Van Tuyl, Adam},
   title={Virtual resolutions of points in $\PP^1\times\PP^1$},
   journal={J. Pure Appl. Algebra},
   volume={226},
   date={2022},
   number={12},
   pages={Paper No. 107140, 18},
}

\bib{reu2019}{article}{
   author={Kenshur, Nathan},
   author={Lin, Feiyang},
   author={McNally, Sean},
   author={Xu, Zixuan},
   author={Yu, Teresa},
   title={On virtually Cohen-Macaulay simplicial complexes},
   journal={J. Algebra},
   volume={631},
   date={2023},
   pages={120--135},
}

\bib{loper}{article}{
   author={Loper, Michael C.},
   title={What makes a complex a virtual resolution?},
   journal={Trans. Amer. Math. Soc. Ser. B},
   volume={8},
   date={2021},
   pages={885--898},
}

\bib{lorenzini}{article}{
    author = {Lorenzini, Anna},
     title = {The minimal resolution conjecture},
   journal = {J. Algebra},
    volume = {156},
      year = {1993},
    number = {1},
     pages = {5--35},
}

\bib{MS}{article}{
    AUTHOR = {Maclagan, Diane},
    author = {Smith, Gregory G.},
     TITLE = {Multigraded {C}astelnuovo-{M}umford regularity},
   JOURNAL = {J. Reine Angew. Math.},
    VOLUME = {571},
      YEAR = {2004},
     PAGES = {179--212},
}

\bib{M2}{misc}{
    LABEL = {M2}, 
    AUTHOR = {Grayson, Daniel R.},
    AUTHOR = {Stillman, Michael E.},
     TITLE = {Macaulay2, a software system for research in algebraic geometry},
    JOURNAL = {Available at \url{http://www.math.uiuc.edu/Macaulay2/}}
}


\bib{mustata}{incollection}{
    AUTHOR = {Musta\c{t}\v{a}, Mircea},
     TITLE = {Graded {B}etti numbers of general finite subsets of points on
              projective varieties},
      NOTE = {Pragmatic 1997 (Catania)},
   JOURNAL = {Matematiche (Catania)},
    VOLUME = {53},
      YEAR = {1998},
     PAGES = {53--81},
}


\bib{virtual-cellular}{article}{
    author={Van Tuyl, Adam}, 
    author={Yang, Jay}, 
    title={Conditions for Virtually Cohen--Macaulay Simplicial Complexes}, 
    journal={\textsf{arxiv:2311.17806 [math.AC]}}, 
    pages={15 pages}, 
}

\bib{monomial-algebras-book}{book}{
   author={Villarreal, Rafael H.},
   title={Monomial algebras},
   series={Monographs and Research Notes in Mathematics},
   edition={2},
   publisher={CRC Press, Boca Raton, FL},
   date={2015},
   pages={xviii+686},
   }

\bib{yang-monomial}{article}{
   author={Yang, Jay},
   title={Virtual resolutions of monomial ideals on toric varieties},
   journal={Proc. Amer. Math. Soc. Ser. B},
   volume={8},
   date={2021},
   pages={100--111},
}
   
\end{biblist}
\end{bibdiv}
\end{document}